\documentclass[a4paper,10pt]{amsart}
\usepackage[utf8]{inputenc}
\usepackage{amsxtra}
\usepackage{amsopn}
\usepackage{amsmath,amsthm,amssymb}
\usepackage{amscd}
\usepackage{amsfonts}
\usepackage{latexsym}
\usepackage{verbatim}
\usepackage{color}

\theoremstyle{plain}
\newtheorem{theorem}{Theorem}[section]
\newtheorem*{theorem*}{Theorem}

\newtheorem{cor}[theorem]{Corollary}

\newtheorem{ex}[theorem]{Example}
\newtheorem*{mt*}{Main Theorem}


\newcommand{\del}{\partial}
\newcommand{\delbar}{\overline{\del}}

\DeclareMathOperator{\Ker}{Ker}

\title[Almost-complex invariants of families of six-dimensional ...]{Almost-complex invariants of families of six-dimensional solvmanifolds}
\author[Nicoletta Tardini and Adriano Tomassini]{Nicoletta Tardini and Adriano Tomassini}
\address{Dipartimento di Scienze Matematiche, Fisiche e Informatiche\\
Unit\`{a} di Matematica e Informatica,
Universit\`{a} degli Studi di Parma\\
Parco Area delle Scienze 53/A, 43124 \\
Parma, Italy}
\email{nicoletta.tardini@gmail.com}
\email{nicoletta.tardini@unipr.it}
\email{adriano.tomassini@unipr.it}

\keywords{almost-complex structure; almost-K\"ahler structure; Hodge numbers.}
\subjclass[2010]{53C15; 58A14; 58J05}
\begin{document}

\maketitle

\begin{abstract}
We compute almost-complex invariants $h^{p,0}_{\delbar}$, $h^{p,0}_{\text{Dol}}$ and almost-Hermitian invariants $h^{p,0}_{\bar\delta}$ on families of almost-K\"ahler and almost-Hermitian $6$-dimensional solvmanifolds. Finally, as a consequence of almost-K\"ahler identities we provide an obstruction to the existence of a symplectic structure on a given  compact  almost-complex manifold. Notice that, when $(X,J,g,\omega)$ is a compact almost Hermitian manifold of real dimension greater than four, not much is known concerning the numbers $h^{p,q}_{\delbar}$.
\end{abstract}


\section{Introduction}
Let $(X, J)$ be a complex manifold, then the Dolbeault cohomology of $X$
$$
H^{\bullet,\bullet}_{\delbar}(X):=
\frac{\text{Ker}\,\delbar}{\text{Im}\,\delbar}
$$
is well defined and it represents an important holomorphic invariant for the complex manifold. If we drop the integrability assumption on $J$, then $\delbar^2\neq 0$ and such a cohomology is not well defined anymore.\\
However, if we fix a $J$-Hermitian metric $g$ on an almost-complex manifold $(X,J)$ and with $*$ we denote the associated Hodge-$*$-operator, then
$$
\Delta_{\delbar}:=\delbar\,\delbar^*+\delbar^*\delbar
$$
is a well-defined second order, elliptic, differential operator. In particular, if $X$ is compact, then $\text{Ker}\Delta_{\delbar}$
is a finite-dimensional complex vector space and we will denote as usual with $h^{\bullet,\bullet}_{\delbar}$ its dimension.
 If $J$ is integrable, then
$$
H^{\bullet,\bullet}_{\delbar}(X)\simeq \text{Ker}\Delta_{\delbar}\,,
$$
and in particular the dimension of the space of harmonic forms depends only on the complex structure and not on the choice of the Hermitian metric.
In \cite[Problem 20]{hirzebruch} Kodaira and Spencer asked whether this is the case also when $J$ is not integrable. More precisely,
\vskip.2truecm \noindent
{\bf Question I} {\em Let $(M,J)$ be an almost complex manifold. Choose an Hermitian metric on $(M,J)$ and consider the numbers $h^{p,q}_{\delbar}$. Is $h^{p,q}_{\delbar}$ independent of the choice of the Hermitian metric?}
\vskip.1truecm \noindent
In \cite{holt-zhang} Holt and Zhang answered negatively to this question, showing with an explicit example that there exist almost complex structures on the Kodaira-Thurston manifold with Hodge number $h^{0,1}_{\delbar}$ varying with different choices of Hermitian metrics. \\
They also proved that if $(M,J,g,\omega)$ is a $4$-dimensional compact almost-K\"ahler manifold, then $h^{1,1}_{\delbar}=b_-+1$, where $b_-$ denotes the dimension of the space of anti self-dual harmonic forms, namely in such  a case $h^{1,1}_{\delbar}$ has a cohomological meaning. In this context, (see \cite[Question 6.2]{holt-zhang}) they asked the following
\vskip.2truecm \noindent
{\bf Question II} {\em Let $(M,J)$ be an almost complex $4$-manifold which admits an almost K\"ahler structure. Does it have a non almost K\"ahler Hermitian metric such that $h^{1,1}_{\delbar}\neq b_-+1$?}. 
\vskip.1truecm \noindent
About this, in \cite[Theorem 3.7]{tardini-tomassini-1} it is proved that if $g$ is a strictly locally conformally K\"ahler metric on a  $4$-dimensional compact almost complex manifold $(X,J)$, then $h^{1,1}_{\delbar}=b_-$. Therefore, since in the non integrable case almost-K\"ahler metrics and strictly locally conformally K\"ahler metrics can coexist, this gives a positive answer to Question II.
\newline
However, when $(X,J,g)$ is a compact almost Hermitian manifold of real dimension greater than four, not much is known concerning the numbers $h^{p,q}_{\delbar}$ and this may be due also by the lack of explicit  computations of such numbers in the literature.
\newline
As a general fact, in special bidegree $(p,0)$, $h^{p,0}_{\delbar}$ is independent of the choice of the Hermitian metric, indeed in this case being $\delbar$-harmonic is equivalent to be $\delbar$-closed. So, in particular $h^{p,0}_{\delbar}$ is a genuine almost-complex invariant.\\
Notice that $h^{n,0}$ is related to the computation of the \emph{Kodaira dimension} of $2n$-dimensional almost-complex manifolds, recently introduced by H. Chen and W. Zhang in \cite{chen-zhang}. For explicit computations of the Kodaira dimension one can refer to \cite{chen-zhang} for the Kodaira-Thurston manifold and to \cite{cattaneo-nannicini-tomassini},  \cite{cattaneo-nannicini-tomassini-1} for several $6$-dimensional solvmanifolds and $4$-dimensional solvmanifolds with no complex structures.\\
In this paper we will compute explicitly the numbers $h^{p,0}_{\delbar}$, for  $p=1,2,3$, on families of six-dimensional manifolds. More in detail, we will consider a family of completely solvable $6$-dimensional solvmanifolds constructed in \cite{fernandez-deleon-saralegui} which is particularly interesting because it admits invariant symplectic structures and invariant almost-complex structures but it does not admit any integrable invariant complex structures. For this reason, in such a case, the computation of these almost-complex invariants is particularly meaningful.
We will consider on such  manifolds an invariant family of almost-K\"ahler structures and we will compute $h^{p,0}_{\delbar}$, with $p=1,2,3$. Furthermore, we will show that these numbers, differently from the integrable case, can vary when the almost-complex structures are almost-K\"ahler and vary continuously (cf. \cite{holt-zhang}).\\
In fact, we will also construct an almost-complex structure which does not admit any compatible symplectic structure and compute $h^{p,0}_{\delbar}$ in this case.\\
Another example will be provided by the computations of $h^{p,0}_{\delbar}$, with $p=1,2,3$ for an almost-K\"ahler structure on the Iwasawa manifold.

Moreover, denoting with $\mu$ the $(2,-1)$-component of the exterior derivative $d$, in \cite{tardini-tomassini} we considered the following differential operator (cf. also \cite{debartolomeis-tomassini})
$$
\bar\delta:=\delbar+\mu
$$
and studied the corresponding harmonic forms. In particular, we compute on the aforementioned families of almost-Hermitian manifolds the $\bar\delta$-harmonic forms of bidegree $(p,0)$.\\
One should notice that the spaces of $\delbar$-harmonic and $\bar\delta$-harmonic forms on non-integrable almost-complex manifolds do not have a cohomological counterpart. However, in \cite{cirici-wilson-1} J. Cirici and S. O. Wilson introduced a generalization of the Dolbeault cohomology on almost-complex manifolds constructing therefore new invariants in this setting. 
By \cite{coelho-placini-stelzig} these cohomology groups on compact almost-complex manifolds are not finite dimensional in general. This means that we have a deep gap between Hodge theory and cohomological theory on almost-complex manifolds.
However, as noticed in \cite{cirici-wilson-1}, in special bi-degrees, e.g., $(p,0)$, the almost-complex Dolbeault cohomology groups have finite dimensions. For this reason, we compute such groups in bi-degree $(p,0)$, for the families of almost-complex manifolds considered above.\\
The paper is organized as follows: in Section \ref{preliminaries} we start by fixing some notations and recalling the basic facts of almost-complex geometry used in the rest of the paper. In Section \ref{section:ak-family} we construct families of almost-K\"ahler solvmanifolds with no left invariant complex structures and then we compute several numerical almost-complex and almost-Hermitian invariants on them. The basic tools to compute the space of harmonic $(p,0)$-forms are suitable Fourier expansions series adapted to the lattices of the solvmanifolds. In Sections \ref{section:no-ak} and \ref{section:iwasawa} we perform similar computations respectively on the same differentiable manifold endowed with an almost-complex structure that does not admit any compatible symplectic structures and on the Iwasawa manifold endowed with an almost-K\"ahler structure. Finally, we apply harmonic theory to give an obstruction to the existence of compatible symplectic structures on almost-complex manifolds.

\medskip
\noindent{\sl Acknowledgments.} The authors would like to thank Luca Lorenzi for useful discussions on elliptic differential operators. They also want to thank Weiyi Zhang for useful suggestions and remarks. 

\section{Preliminaries}\label{preliminaries}
In this Section we recall some basic facts about almost-complex manifolds and fix some notations.
Let $X$ be a smooth manifold of dimension $2n$ and let $J$ be an almost-complex structure on $X$, i.e., a $(1,1)$-tensor on $X$ such that $J^2=-\text{Id}$. Then, $J$ induces a natural bigrading on the space of complex valued differential forms $A^\bullet(X)$, namely
$$
A^\bullet(X)=\bigoplus_{p+q=\bullet}A^{p,q}(X)\,.
$$
According to this decomposition, the exterior derivative $d$ splits into four operators
$$
d:A^{p,q}(X)\to A^{p+2,q-1}(X)\oplus A^{p+1,q}(X)\oplus A^{p,q+1}(X)\oplus A^{p-1,q+2}(X)
$$
$$
d=\mu+\del+\delbar+\bar\mu\,,
$$
where $\mu$ and $\bar\mu$ are differential operators that are linear over functions. 
The almost-complex structure $J$ is integrable, that is $J$ induces a complex structure on $X$, if and only if $\mu=\bar\mu=0$.\\
In general, since $d^2=0$, one has the following relations
$$
\left\lbrace
\begin{array}{lcl}
\mu^2 & =& 0\\
\mu\del+\del\mu & = & 0\\
\del^2+\mu\delbar+\delbar\mu & = & 0\\
\del\delbar+\delbar\del+\mu\bar\mu+\bar\mu\mu & = & 0\\
\delbar^2+\bar\mu\del+\del\bar\mu & = & 0\\
\bar\mu\delbar+\delbar\bar\mu & = & 0\\
\bar\mu^2 & =& 0
\end{array}
\right.\,
$$
and so the Dolbeault cohomology of $X$
$$
H^{\bullet,\bullet}_{\delbar}(X):=
\frac{\text{Ker}\,\delbar}{\text{Im}\,\delbar}
$$
is well defined if and only if $J$ is integrable.\\
If $g$ is an Hermitian metric on $(X,J)$ with associated fundamental form $\omega$ and $*$ denotes the Hodge-$*$-operator, one can consider the following differential operator
$$
\Delta_{\delbar}:=\delbar\,\delbar^*+\delbar^*\delbar\,.
$$
This is a second order, elliptic, differential operator and we will denote its kernel by
$$
\mathcal{H}^{p,q}_{\delbar}(X):=\text{Ker}\,\Delta_{\delbar_{\vert A^{p,q}(X)}}\,.
$$
If $X$ is compact this space is finite-dimensional and its dimension will be denoted by $h^{p,q}_{\delbar}(X)$. By \cite{holt-zhang} we know that these Hodge numbers are not almost-complex invariants, more precisely they depend on the choice of the Hermitian metric. \\
In \cite{tardini-tomassini} we considered the following differential operator (cf. also \cite{debartolomeis-tomassini})
$$
\bar\delta:=\delbar+\mu
$$
and we set
$$
\Delta_{\bar\delta}:=\bar\delta\bar\delta^*+\bar\delta^*\bar\delta\,.
$$
This is a second order, elliptic, differential operator and we denote with 
$$
\mathcal{H}^k_{\bar\delta}(X):=\Ker\Delta_{\bar\delta_{\vert A^k(X)}}
$$
the space of $\bar\delta$-harmonic $k$-forms and with
$$
\mathcal{H}^{p,q}_{\bar\delta}(X):=\Ker\Delta_{\bar\delta_{\vert A^{p,q}(X)}}
$$
the space of $\bar\delta$-harmonic $(p,q)$-forms. If $X$ is compact these spaces are finite dimensional, and we will set $h^k_{\bar\delta}(X)$ and
$h^{p,q}_{\bar\delta}(X)$ for their dimensions respectively.\\
Moreover, if we set
$$
\Delta_{\mu}:=\mu\mu^*+\mu^*\mu\,,
$$
we have that the associated spaces of harmonic forms
$\mathcal{H}^{\bullet,\bullet}_{\mu}(X)$ and $\mathcal{H}^{\bullet}_{\mu}(X)$ are infinite-dimensional in general, indeed  $\mu$ is linear over functions.\\
In \cite[Proposition 5.5]{tardini-tomassini} we showed that on a compact almost-Hermitian manifold $(X,J,g)$ we have
$$
\mathcal{H}^{\bullet}_{\delbar}(X)\cap
\mathcal{H}^{\bullet}_{\mu}(X)\subseteq
\mathcal{H}^{\bullet}_{\bar\delta}(X)
$$
and on bi-graded forms we have the equality (cf. \cite[Remark 5.6]{tardini-tomassini})
$$
\mathcal{H}^{\bullet,\bullet}_{\delbar}(X)\cap
\mathcal{H}^{\bullet,\bullet}_{\mu}(X)=
\mathcal{H}^{\bullet,\bullet}_{\bar\delta}(X)\,.
$$

\section{Families of Almost-K\"ahler solvmanifolds with no left-invariant complex structures}\label{section:ak-family}

We recall the following construction from \cite{fernandez-deleon-saralegui}.
Let $G$ be the following connected $2$-step solvable $6$-dimensional Lie group 
$$
G:=\left\lbrace
\left[\begin{matrix}
e^t & 0 &   xe^t & 0 & 0 & y_1\\
0 & e^{-t}   & 0 & xe^{-t} & 0 & y_2\\
0 & 0 & e^t   & 0 &  0 & z_1\\
0 & 0 & 0 &  e^{-t}   &   0 & z_2\\
0 & 0 & 0 & 0 &1 &t\\
0 & 0 & 0 & 0 &0 &1\\
\end{matrix}\right]
\mid y_1,y_2,z_1,z_2,t,x\in\mathbb{R}
\right\rbrace
$$
and set
$$
\left\lbrace
\begin{array}{lcl}
e^1 & =& dt\\
e^2 & = & dx\\
e^3 & = & e^{-t}dy_1-xe^{-t}dz_1\\
e^4 & = & e^{t}dy_2-xe^{t}dz_2\\
e^5 & = & e^{-t}dz_1\\
e^6 & = & e^tdz_2
\end{array}
\right.\,,
$$
for a basis of left-invariant $1$-forms on $G$, and the dual basis is given by
$$
\left\lbrace
\begin{array}{lcl}
e_1 & =& \frac{\partial}{\partial t}\\
e_2 & = & \frac{\partial}{\partial x}\\
e_3 & = & e^{t}\frac{\partial}{\partial y_1}\\
e_4 & = & e^{-t} \frac{\partial}{\partial y_2}\\
e_5 & = & e^{t}\frac{\partial}{\partial z_1}+xe^t\frac{\partial}{\partial y_1}\\
e_6 & = & e^{-t}\frac{\partial}{\partial z_2}+xe^{-t}\frac{\partial}{\partial y_2}
\end{array}
\right.\,,
$$
In particular, the following structure equations hold
$$
\left\lbrace
\begin{array}{lcl}
de^1 & =& 0\\
de^2 & = & 0\\
de^3 & = & -e^{13}-e^{25}\\
de^4 & = & e^{14}-e^{26}\\
de^5 & = & -e^{15}\\
de^6 & = & e^{16}
\end{array}
\right.\,,
$$
where, as usual, we set $e^{ij}:=e^i\wedge e^j$, and
$$
[e_1,e_3]=[e_2,e_5]=e_3,\quad
[e_1,e_4]=-[e_2,e_6]=-e_4,\quad
[e_1,e_5]=e_5,\quad
[e_1,e_6]=-e_6\,.
$$
Let $\mathfrak{g}$ be the Lie algebra of $G$, then $\mathfrak{g}$ is completely solvable. In fact, $G$ can be seen as a semidirect product $G=\mathbb{R}^2\ltimes_\Phi\mathbb{R}^4$, where for every $(t,x)\in\mathbb{R}^2$,
$$
\Phi(t,x):\mathbb{R}^4\to\mathbb{R}^4,\quad
\Phi(t,x)=\left[\begin{matrix}
e^t & 0 &   xe^t & 0 \\
0 & e^{-t}   & 0 & xe^{-t} \\
0 & 0 & e^t   & 0 \\
0 & 0 & 0 &  e^{-t} \\
\end{matrix}\right]
$$
and the group operation on $G$ is given by
$$
\left(t,x,y_1,y_2,z_1,z_2\right)*\left(t',x',y'_1,y'_2,z'_1,z'_2\right)=
$$
$$
\left(t+t',x+x',y'_1e^t+xz'_1e^t+y_1,
y'_2e^{-t}+xz'_2e^{-t}+y_2,
z'_1e^t+z_1,z'_2e^{-t}+z_2\right)\,.
$$
A lattice $\Gamma$ for $G$ can be constructed as follows.
Let $B\in SL(2,\mathbb{Z})$ be a unimodular matrix with integer entries and distinct eigenvalues $e^{a_0}$, $e^{-a_0}$. Then there exists a real invertible matrix $P$ such that
$$
PBP^{-1}=
\left[\begin{matrix}
e^{a_0} & 0  \\
0 & e^{-a_0}\\
\end{matrix}\right]\,.
$$
Let $\tilde\Gamma:=a_0\mathbb{Z}\times\mathbb{Z}$ and
$L:=\left((m_1,m_2)P^t,(n_1,n_2)P^t\right)$ with $m_1,m_2,n_1,n_2\in\mathbb{Z}$. Then,
$\Gamma:=\tilde\Gamma\ltimes_\Phi L$ is a lattice in $G$ and we set
$X:=\Gamma\backslash G$ for the associated solvmanifold. In fact, $X$ has the structure of a $\mathbb{T}^4$-bundle over $\mathbb{T}^2$.\\
As proven in \cite{fernandez-deleon-saralegui}, $X$ is a completely solvable solvmanifold which admits symplectic structures but none of them satisfies the Hard Lefschetz condition. Moreover, $X$ is not formal but all the triple Massey products vanish. Finally, $X$ does not admit any invariant integrable complex structure.\\

%

Now we construct a family of left-invariant almost-complex structures on $X$. As noticed in \cite{fernandez-deleon-saralegui} the arbitrary left-invariant symplectic structure on $X$ is given by
$$
\omega_{a,b,c}=ae^{12}+be^{56}+c(e^{36}+e^{45})
$$
with $a,b,c\in\mathbb{R}$ and $a,c\neq 0$. We define the following compatible almost-complex structure $J_{a,b,c}$,
$$
\left\lbrace
\begin{array}{lcl}
J_{a,b,c}e_1 & =& ae_2\\
J_{a,b,c}e_2 & = & -\frac{1}{a}e_1\\
J_{a,b,c}e_3 & = & ce_6\\
J_{a,b,c}e_4 & = & ce_5-be_3\\
J_{a,b,c}e_5 & = & -\frac{1}{c}e_4+be_6\\
J_{a,b,c}e_6 & = & -\frac{1}{c}e_3
\end{array}
\right.\,,
$$
and it acts on forms by
$$
\left\lbrace
\begin{array}{lcl}
J_{a,b,c}e^1 & =& -\frac{1}{a}e^2\\
J_{a,b,c}e^2 & = & ae^1\\
J_{a,b,c}e^3 & = & -be^4-\frac{1}{c}e^6\\
J_{a,b,c}e^4 & = & -\frac{1}{c}e^5\\
J_{a,b,c}e^5 & = & ce^4\\
J_{a,b,c}e^6 & = & be^5+ce^3
\end{array}
\right.\,.
$$
Hence, $(J_{a,b,c},\omega_{a,b,c})$ is a family of left-invariant almost-K\"ahler structures on $X$. \\
A global co-frame of $(1,0)$-forms is provided by
$$
\varphi^1:=ae^1+ie^2\,,\qquad
\varphi^2:=be^5+ce^3+ie^6\,,\qquad
\varphi^3:=ce^4+ie^5\,,
$$
and the dual frame of $(1,0)$-vectors is given by
$$
V_1:=\frac{1}{2}\left(\frac{1}{a}e_1-ie_2\right)\,,\qquad
V_2:=\frac{1}{2}\left(\frac{1}{c}e_3-ie_6\right)\,,\qquad
V_3:=\frac{1}{2}\left(\frac{1}{c}e_4-ie_5+i\frac{b}{c}e_3\right)\,.
$$
In particular, the complex structure equations become
$$
\left\lbrace
\begin{array}{lcl}
d\varphi^1 & =& 0\\
d\varphi^2 & = & \frac{c}{4}\varphi^{13}-\frac{1}{2a}\varphi^{1\bar 2}-\frac{c}{4}\varphi^{1\bar3}+\frac{c}{4}\varphi^{3\bar 1}-
\frac{1}{2a}\varphi^{\bar1\bar2}+\frac{c}{4}\varphi^{\bar1\bar3}\\
d\varphi^3 & = & \frac{c}{4}\varphi^{12}-\frac{c}{4}\varphi^{1\bar 2}+\frac{1}{2a}\varphi^{1\bar3}+\frac{c}{4}\varphi^{2\bar 1}+
\frac{c}{4}\varphi^{\bar1\bar2}+\frac{1}{2a}\varphi^{\bar1\bar3}\\
\end{array}
\right.\,.
$$

\section{Numerical almost-complex and almost-Hermitian invariants on $(X,J_{a,b,c},\omega_{a,b,c})$}

In this section we compute several almost-complex invariants on $(X,J_{a,b,c},\omega_{a,b,c})$. In particular, we start with the Hodge numbers $h^{p,0}_{\delbar}$, with $p=1,2,3$.

%

\subsection{Computations for $\mathcal{H}^{3,0}_{\delbar}$}

We compute now $\mathcal{H}^{3,0}_{\delbar}$ for $X:=(X,J_{a,b,c},\omega_{a,b,c})$. Let
$$
\psi=A\varphi^{123}
$$
with $A$ smooth function on $X$, be an arbitrary $(3,0)$-form on $X$.
By degree reasons, $\psi$ is $\delbar$-harmonic if and only if $\delbar\psi=0$.
Since $\varphi^{123}$ is $\delbar$-closed we have
$$
\delbar\psi=-\bar V_1(A)\varphi^{123\bar 1}-\bar V_2(A)\varphi^{123\bar 2}-\bar V_3(A)\varphi^{123\bar 3},
$$
hence $\delbar\psi=0$ if and only if
$$
\bar V_1(A)=\bar V_2(A)=\bar V_3(A)=0
$$
hence $(V_1\bar V_1+V_2\bar V_2+V_3\bar V_3)(A)=0$ and, since $V_1\bar V_1+V_2\bar V_2+V_3\bar V_3$ is an elliptic operator we have that $A$ is constant. Therefore,
$$
\mathcal{H}^{3,0}_{\delbar}(X)=\left\langle\varphi^{123}\right\rangle
$$
and $h^{3,0}_{\delbar}=1$.

\subsection{Computations for $\mathcal{H}^{1,0}_{\delbar}$}

Let
$$
\psi=A\varphi^{1}+B\varphi^2+D\varphi^3
$$
with $A,B,D$ smooth functions on $X$, be an arbitrary $(1,0)$-form on $X$.
By degree reasons, $\psi$ is $\delbar$-harmonic if and only if $\delbar\psi=0$.
Using the structure equations we have
$$
\delbar\psi=-\bar V_1(A)\varphi^{1\bar1}-\bar V_2(A)\varphi^{1\bar2}-
\bar V_3(A)\varphi^{1\bar3}
-\bar V_1(B)\varphi^{2\bar1}-\bar V_2(B)\varphi^{2\bar2}-
\bar V_3(B)\varphi^{2\bar3}
$$
$$
-\bar V_1(D)\varphi^{3\bar1}-\bar V_2(D)\varphi^{3\bar2}-
\bar V_3(D)\varphi^{3\bar3}-
\frac{B}{2a}\varphi^{1\bar2}-\frac{1}{4}B\varphi^{1\bar3}+
B\frac{c}{4}\varphi^{3\bar1}-\frac{c}{4}D\varphi^{1\bar2}+
\frac{1}{2a}D\varphi^{1\bar 3}+\frac{c}{4}D\varphi^{2\bar 1},
$$
hence $\delbar\psi=0$ if and only if
$$
\left\lbrace
\begin{array}{lcl}
\bar V_1(A) & =& 0\\
\bar V_2(A)+\frac{1}{2a}B+\frac{c}{4}D & = & 0\\
\bar V_3(A)+\frac{1}{4}B-\frac{1}{2a}D & = & 0\\
\bar V_1(B)-\frac{c}{4}D & = & 0\\
\bar V_2(B) & =& 0\\
\bar V_3(B) & =& 0\\
\bar V_1(D)-\frac{c}{4}B & = & 0\\
\bar V_2(D) & =& 0\\
\bar V_3(D) & =& 0\\
\end{array}
\right.\,.
$$
In particular, by $\bar V_2(B)=\bar V_3(B)=0$ we have that $V_2\bar V_2(B)=V_3\bar V_3(B)=0$ and $V_2\bar V_2+V_3\bar V_3$ is a strictly elliptic operator without zero order terms when $B$ is viewed as function of $y_1,y_2,z_1,z_2$. Since the fiber is compact by the maximum principle $B$ is constant on the fibers, then $B$ is a function on the base with $(t,x)$ as coordinates. Namely, $B=B(t,x)$ and similarly by the previous system, $D=D(t,x)$.\\
As a consequence, from the first three equations
$$
(V_1\bar V_1+V_2\bar V_2+V_3\bar V_3)(A)=0
$$
then $A$ is constant.\\
The previous system reduces to
$$
\left\lbrace
\begin{array}{lcl}
\frac{1}{2a}B+\frac{c}{4}D & = & 0\\
\frac{1}{4}B-\frac{1}{2a}D & = & 0\\
\bar V_1(B)-\frac{c}{4}D & = & 0\\
\bar V_1(D)-\frac{c}{4}B & = & 0\\
\end{array}
\right.\,.
$$
In particular,
$$
B=-\frac{ac}{2}D\,,\qquad\text{and}\qquad
\frac{a^2c+4}{4a}D=0\,.
$$
Therefore we have two cases to consider. First, if $a^2c+4\neq 0$ then
$$
D=0,\quad B=0, \quad A=\text{const}
$$
hence
$$
\mathcal{H}^{1,0}_{\delbar}=\left\langle \varphi^1\right\rangle
$$
and $h^{1,0}_{\delbar}=1$.\\
If $a^2c+4= 0$, since $B=-\frac{ac}{2}D$, the system reduces to
$$
\left\lbrace
\begin{array}{lcl}
\frac{ac}{4}\bar V_1(D)+\frac{c}{4}D & = & 0\\
\bar V_1(D)+\frac{ac^2}{8}D & = & 0\\
\end{array}
\right.\,.
$$
that is
$$
\left\lbrace
\begin{array}{lcl}
\left(-\frac{ac^2}{8}+\frac{1}{2a}\right)D & = & 0\\
\bar V_1(D)+\frac{1}{2a}D & = & 0\\
\end{array}
\right.\,.
$$
By the first equation we have $(-a^2c^2+4)D=0$, and recalling that $a^2c+4= 0$, we have two cases. If $a\neq\pm 2$ then
$$
D=0,\quad B=0, \quad A=\text{const}
$$
hence
$$
\mathcal{H}^{1,0}_{\delbar}=\left\langle \varphi^1\right\rangle
$$
and $h^{1,0}_{\delbar}=1$.\\
If $a=\pm 2$, we are left with
$$
\bar V_1(D)\pm\frac{1}{4}D  =  0,\quad
B=\pm D,
\quad A=\text{const}\,.
$$
Since, $D=D(t,x)$ we can expand in Fourier series and get
$$
D=\sum_{\lambda,\mu\in\mathbb{Z}}D_{\lambda\mu}e^{2\pi i(\lambda x+\frac{\mu}{a_0} t)}
$$
with $D_{\lambda\mu}$ constants for every $\lambda,\mu\in\mathbb{Z}$. The equation $\bar V_1(D)\pm\frac{1}{4}D  =  0$ becomes
$$
(\frac{1}{a}2\pi i \frac{\mu}{a_0}-2\pi\lambda)D_{\lambda\mu}\pm\frac{1}{2}D_{\lambda\mu}=0
$$
namely,
$$
\left((-4\pi\lambda\pm1)+i(4\pi\frac{\mu}{a_0}\frac{1}{a})\right)D_{\lambda\mu}=0
$$
and since $-4\pi\lambda\pm1\neq0$ for every $\lambda\in\mathbb{Z}$ we have that $D_{\lambda\mu}$ for every $\lambda,\mu\in\mathbb{Z}$.
Therefore,
$$
D=0,\quad B=0, \quad A=\text{const}
$$
hence
$$
\mathcal{H}^{1,0}_{\delbar}=\left\langle \varphi^1\right\rangle
$$
and $h^{1,0}_{\delbar}=1$.\\

\subsection{Computations for $\mathcal{H}^{2,0}_{\delbar}$}

Let
$$
\psi=A\varphi^{12}+B\varphi^{13}+D\varphi^{23}
$$
with $A,B,D$ smooth functions on $X$, be an arbitrary $(2,0)$-form on $X$.
By degree reasons, $\psi$ is $\delbar$-harmonic if and only if $\delbar\psi=0$.
Using the structure equations we have
$$
\begin{array}{lll}
\delbar\psi &=& \bar V_1(A)\varphi^{12\bar1}+\bar V_2(A)\varphi^{12\bar2}+
\bar V_3(A)\varphi^{12\bar3}
+\bar V_1(B)\varphi^{13\bar1}+\bar V_2(B)\varphi^{13\bar2}+
\\[5pt]
&{}& +\bar V_3(B)\varphi^{13\bar3}+\bar V_1(D)\varphi^{23\bar1}+\bar V_2(D)\varphi^{23\bar2}+
\bar V_3(D)\varphi^{23\bar3}-
\frac{c}{4}A\varphi^{13\bar1}+\\[5pt]
&{}&-  \frac{c}{4}B\varphi^{12\bar1} +D\frac{1}{2a}\varphi^{13\bar2}+
\frac{c}{4}D\varphi^{13\bar3}-
\frac{c}{4}D\varphi^{12\bar 2}+\frac{1}{2a}D\varphi^{12\bar 3},
\end{array}
$$

hence $\delbar\psi=0$ if and only if
$$
\left\lbrace
\begin{array}{lcl}
\bar V_1(A)-\frac{c}{4}B & =& 0\\
\bar V_2(A)-\frac{c}{4}D & = & 0\\
\bar V_3(A)+\frac{1}{2a}D & = & 0\\
\bar V_1(B)-\frac{c}{4}A & = & 0\\
\bar V_2(B)+\frac{1}{2a}D & =& 0\\
\bar V_3(B) +\frac{c}{4}D& =& 0\\
\bar V_1(D) & = & 0\\
\bar V_2(D) & =& 0\\
\bar V_3(D) & =& 0\\
\end{array}
\right.\,.
$$
From the last two equations we obtain immediately that $D=\text{const}$. Hence, from the system we have that
$$
V_2\bar V_2(A)=V_3\bar V_3(A)=V_2\bar V_2(B)=V_3\bar V_3(B)=0
$$
hence, with a similar argument used before we have that 
$$A=A(t,x), \quad B=B(t,x).
$$
In particular, this implies that
$$
D=0.
$$
We can expand in Fourier series and get
$$
A=\sum_{\lambda,\mu\in\mathbb{Z}}A_{\lambda\mu}e^{2\pi i(\lambda x+\frac{\mu}{a_0} t)},\quad
B=\sum_{\lambda,\mu\in\mathbb{Z}}B_{\lambda\mu}e^{2\pi i(\lambda x+\frac{\mu}{a_0} t)}
$$
with $A_{\lambda\mu},B_{\lambda\mu}$ constants for every $\lambda,\mu\in\mathbb{Z}$. 
The first and fourth equations become respectively
$$
\left(\frac{1}{a}2\pi i\frac{\mu}{a_0}-2\pi\lambda\right)A_{\lambda\mu}-
\frac{c}{2}B_{\lambda\mu}=0
$$
$$
\left(\frac{1}{a}2\pi i\frac{\mu}{a_0}-2\pi\lambda\right)B_{\lambda\mu}-
\frac{c}{2}A_{\lambda\mu}=0.
$$
Summing the two equations we get
$$
\left((-2\pi\lambda-\frac{c}{2})+i(\frac{1}{a}2\pi\frac{\mu}{a_0})\right)(A_{\lambda\mu}+B_{\lambda\mu})=0.
$$
Now we consider two cases: $c\notin 4\pi\mathbb{Z}$ and $c\in 4\pi\mathbb{Z}$.\\
If $c\notin 4\pi\mathbb{Z}$, then $A_{\lambda\mu}+B_{\lambda\mu}=0$ for every $\lambda\mu\in\mathbb{Z}$, implying that $A=-B$. In this case, we obtain the following equation
$$
\bar V_1(A)+\frac{c}{4}A  = 0
$$
and so
$$
\left((-2\pi\lambda+\frac{c}{2})+i(\frac{1}{a}2\pi\frac{\mu}{a_0})\right)A_{\lambda\mu}=0.
$$
Therefore, under our assumption $A_{\lambda\mu}=0$ for every $\lambda,\mu\in\mathbb{Z}$ and therefore $B_{\lambda\mu}=0$ for every $\lambda,\mu\in\mathbb{Z}$.
As a consequence we have that if $c\notin 4\pi\mathbb{Z}$,
$$
A=0,\quad B=0, \quad D=0
$$
hence
$$
\mathcal{H}^{2,0}_{\delbar}=0
$$
and $h^{2,0}_{\delbar}=0$.\\
If $c\in 4\pi\mathbb{Z}$, we set $c=4\pi k$ with $k\in\mathbb{Z}\setminus\left\lbrace 0\right\rbrace$, since by construction $c\neq 0$. The equation becomes
$$
\left((-2\pi\lambda-2\pi k)+i(\frac{1}{a}2\pi\frac{\mu}{a_0})\right)(A_{\lambda\mu}+B_{\lambda\mu})=0.
$$
If $(\lambda,\mu)\neq (-k,0)$ then $A_{\lambda\mu}+B_{\lambda\mu}=0$, otherwise the equation is trivially satisfied.\\
Suppose that $(\lambda,\mu)\neq (-k,0)$, then $A_{\lambda\mu}=-B_{\lambda\mu}$ and the first equation becomes
$$
\left((-2\pi\lambda+2\pi k)+i(\frac{1}{a}2\pi\frac{\mu}{a_0})\right)A_{\lambda\mu}=0.
$$
Hence, if, moreover $(\lambda,\mu)\neq (k,0)$ then $A_{\lambda\mu}=-B_{\lambda\mu}=0$. Namely, resuming we have that
\begin{itemize}
\item $A_{\lambda\mu}=B_{\lambda\mu}=0$ if $(\lambda,\mu)\neq (\pm k,0)$
\item $A_{k0}=-B_{k0}=0$
\item we have no informations on $A_{-k0},B_{-k0}$.
\end{itemize}
The Fourier expansions reduces to
$$
A=A_{k0}e^{2\pi ikx}+A_{-k0}e^{-2\pi ikx}
$$
and
$$
B=-A_{k0}e^{2\pi ikx}+B_{-k0}e^{-2\pi ikx}\,.
$$
In particular, the equation $\bar V_1(A)-\frac{c}{4}B=0$ becomes
$$
2\pi k(A_{-k0}-B_{-k0})e^{-2\pi ikx}=0
$$
giving $A_{-k0}=B_{-k0}$, and also the other equations are now satisfied.
Therefore,
$$
A=A_{k0}e^{2\pi ikx}+A_{-k0}e^{-2\pi ikx},\quad
B=-A_{k0}e^{2\pi ikx}+A_{-k0}e^{-2\pi ikx},\quad
D=0
$$
satisfy the system of equations for $\mathcal{H}^{2,0}_{\delbar}$ hence, if $c\in 4\pi\mathbb{Z}$, $c\neq 0$, $h^{2,0}_{\delbar}=2$.

\medskip

Therefore, we just proved the following
\begin{theorem}
Let $(X,J_{a,b,c},\omega_{a,b,c})$ be the family of almost-K\"ahler manifolds previously constructed. Then,
\begin{itemize}
\item $h^{1,0}_{\delbar}=1$,
\item $h^{2,0}_{\delbar}=
\left\lbrace
\begin{array}{lcl}
0 & \text{if }& c\notin 4\pi\mathbb{Z}\\
2 & \text{if }& c\in 4\pi\mathbb{Z}\\
\end{array}
\right.\,,$
\item $h^{3,0}_{\delbar}=1$.
\end{itemize}
\end{theorem}

An immediate consequence is the following result that marks a difference with the integrable case (cf. also \cite{holt-zhang}).
\begin{cor}
The Hodge numbers can vary when the almost-complex structures are almost-K\"ahler and vary continuously.
\end{cor}

\medskip
\noindent
We compute now the almost-Hermitian invariants $h^{p,0}_{\bar\delta}$, with $p=1,2,3$.


\noindent
First of all we recall that on bi-graded forms
$\mathcal H^{\bullet,\bullet}_{\bar\delta}=\mathcal H^{\bullet,\bullet}_{\delbar}\cap \mathcal H^{\bullet,\bullet}_{\mu}$, in particular for bidegree reasons 
$$
\mathcal H^{1,0}_{\bar\delta}=\mathcal H^{1,0}_{\delbar}\,,
$$
hence we are left to compute $\mathcal H^{2,0}_{\bar\delta}$ and $\mathcal H^{3,0}_{\bar\delta}$.

\subsection{Computations for $\mathcal{H}^{3,0}_{\bar\delta}$}

It is immediate to see that
$$
\mathcal{H}^{3,0}_{\bar\delta}=\mathcal{H}^{3,0}_{\delbar}\cap
\text{Ker}\,(\mu^*)\,.
$$
Since $\mathcal{H}^{3,0}_{\delbar}=\langle\varphi^{123}\rangle$ we set
$\psi=A\varphi^{123}$ with $A\in\mathbb{C}$. Then, $\psi\in \text{Ker}\,(\mu^*)$ if and only if $\bar\mu*\psi=0$.
Since $*\psi=A\cdot\text{const}\cdot\varphi^{123}$ and, by the structure equation
$$
\bar\mu\varphi^{123}=
\frac{1}{2a}\varphi^{13\bar 1\bar 2}-\frac{c}{4}\varphi^{13\bar1\bar3}+
\frac{c}{4}\varphi^{12\bar1\bar2}+\frac{1}{2a}\varphi^{12\bar1\bar3},
$$
we have that $\bar\mu*\psi=0$ if and only if $A=0$. Therefore,
$$
\mathcal H^{3,0}_{\bar\delta}=\left\lbrace 0\right\rbrace
$$
and $h^{3,0}_{\bar\delta}=0$.

\subsection{Computations for $\mathcal{H}^{2,0}_{\bar\delta}$}

It is immediate to see that
$$
\mathcal{H}^{2,0}_{\bar\delta}=\mathcal{H}^{2,0}_{\delbar}\cap
\text{Ker}\,(\mu^*)\,.
$$
If $c\notin 4\pi\mathbb Z$ then $\mathcal{H}^{2,0}_{\delbar}=\left\lbrace 0\right\rbrace$, hence $\mathcal{H}^{2,0}_{\bar\delta}=\left\lbrace 0\right\rbrace$.\\
Let us assume that $c\in 4\pi\mathbb Z$, namely $c=4\pi k$, with $k\in\mathbb Z\setminus\left\lbrace 0\right\rbrace$.\\
Since 
$$
\mathcal{H}^{2,0}_{\delbar}=\left\langle
e^{2\pi ikx}\varphi^{12}-e^{2\pi ikx}\varphi^{13},
e^{-2\pi ikx}\varphi^{12}+e^{-2\pi ikx}\varphi^{13}
\right\rangle
$$
We set
$$
\psi=A(e^{2\pi ikx}\varphi^{12}-e^{2\pi ikx}\varphi^{13})+B
(e^{-2\pi ikx}\varphi^{12}+e^{-2\pi ikx}\varphi^{13})
$$
 with $A,B\in\mathbb{C}$. Then, $\psi\in \text{Ker}\,(\mu^*)$ if and only if $\bar\mu*\psi=0$.\\
 Since
 $$
 *\varphi^{12}=\frac{i}{2}\varphi^{123\bar 3}\,,\qquad
  *\varphi^{13}=-\frac{i}{2}\varphi^{123\bar 2}
 $$
we have that
$$
*\psi=A\frac{i}{2}(e^{2\pi ikx}\varphi^{123\bar3}+e^{2\pi ikx}\varphi^{123\bar2})+B\frac{i}{2}
(e^{-2\pi ikx}\varphi^{123\bar3}-e^{-2\pi ikx}\varphi^{123\bar 2})\,.
$$
By the structure equations
$$
\bar\mu\varphi^{123\bar2}=\frac{c}{4}\varphi^{13\bar1\bar2\bar3}-
\frac{1}{2a}\varphi^{12\bar1\bar2\bar3}\,,
\quad
\bar\mu\varphi^{123\bar3}=\frac{1}{2a}\varphi^{13\bar1\bar2\bar3}+
\frac{c}{4}\varphi^{12\bar1\bar2\bar3}\,.
$$
Hence, we obtain
$$
\bar\mu*\psi=
\varphi^{12\bar1\bar2\bar3}\left[
A\frac{i}{2}(\frac{c}{4}-\frac{1}{2a})e^{2\pi ikx}+
B\frac{i}{2}(\frac{c}{4}+\frac{1}{2a})e^{-2\pi ikx}\right]+
$$
$$
\varphi^{13\bar1\bar2\bar3}\left[
A\frac{i}{2}(\frac{c}{4}+\frac{1}{2a})e^{2\pi ikx}+
B\frac{i}{2}(\frac{1}{2a}-\frac{c}{4})e^{-2\pi ikx}\right]\,.
$$
Therefore, $\bar\mu*\psi=0$ if and only if
$$
A(\frac{c}{4}-\frac{1}{2a})e^{4\pi ikx}+
B(\frac{c}{4}+\frac{1}{2a})=0\,,
$$
and
$$
A(\frac{c}{4}+\frac{1}{2a})e^{4\pi ikx}+
B(\frac{1}{2a}-\frac{c}{4})=0.
$$
This implies that $A=B=0$, namely $\psi=0$.\\
Therefore,
$$
\mathcal H^{2,0}_{\bar\delta}=\left\lbrace 0\right\rbrace
$$
and $h^{2,0}_{\bar\delta}=0$.\\

\medskip

Therefore, we just proved the following
\begin{theorem}
Let $(X,J_{a,b,c},\omega_{a,b,c})$ be the family of almost-K\"ahler manifolds previously constructed. Then,
\begin{itemize}
\item $h^{1,0}_{\bar\delta}=1$,
\item $h^{2,0}_{\bar\delta}=0$,
\item $h^{3,0}_{\bar\delta}=0$.
\end{itemize}
\end{theorem}


\medskip

Now we compute the dimension of the almost-complex Dolbeault cohomology groups $H^{p,0}_{\text{Dol}}$.\\

First of all, notice that by \cite[Proposition 4.10]{cirici-wilson-1}, 
$$
H^{p,0}_{\text{Dol}}\simeq \mathcal{H}^{p,0}_{\delbar}\cap\text{Ker}\,\bar\mu
$$

\subsection{Computation of $H^{1,0}_{\text{Dol}}$ and $H^{3,0}_{\text{Dol}}$}

Clearly, by the structure equations and by the previous computations

$$
H^{1,0}_{\text{Dol}}\simeq \mathcal{H}^{1,0}_{\delbar}\cap\text{Ker}\,\bar\mu=\left\langle\varphi^1\right\rangle\,.
$$
Now, since $\mathcal{H}^{3,0}_{\delbar}=\left\langle\varphi^{123}\right\rangle$
and by a direct computation $\bar\mu\varphi^{123}\neq 0$, one has that
$$
H^{3,0}_{\text{Dol}}=\left\lbrace 0\right\rbrace.
$$

\subsection{Computation of $H^{2,0}_{\text{Dol}}$}

Notice that, if $c\notin 4\pi\mathbb{Z}$, then $\mathcal{H}^{2,0}_{\delbar}=\left\lbrace 0\right\rbrace$ and so
$$
H^{2,0}_{\text{Dol}}=\left\lbrace 0\right\rbrace.
$$
Let now $c\in 4\pi\mathbb{Z}$, then
$$
\mathcal{H}^{2,0}_{\delbar}=\left\langle
e^{2\pi ikx}\varphi^{12}-e^{2\pi ikx}\varphi^{13},
e^{-2\pi ikx}\varphi^{12}+e^{-2\pi ikx}\varphi^{13}
\right\rangle
$$
We set
$$
\psi=A(e^{2\pi ikx}\varphi^{12}-e^{2\pi ikx}\varphi^{13})+B
(e^{-2\pi ikx}\varphi^{12}+e^{-2\pi ikx}\varphi^{13})
$$
 with $A,B\in\mathbb{C}$. Since
 $$
 \bar\mu\varphi^{12}=\frac{1}{2a}\varphi^{1\bar1\bar2}-\frac{c}{4}\varphi^{1\bar1\bar3}\,,\quad
  \bar\mu\varphi^{13}=-\frac{c}{4}\varphi^{1\bar1\bar2}-\frac{1}{2a}\varphi^{1\bar1\bar3}\,,
 $$
 then, $\bar\mu\psi=0$ if and only if
 $$
A(\frac{c}{4}+\frac{1}{2a})e^{4\pi ikx}+
B(\frac{1}{2a}-\frac{c}{4})=0.
$$
and
 $$
A(-\frac{c}{4}+\frac{1}{2a})e^{4\pi ikx}+
B(-\frac{c}{4}-\frac{1}{2a})=0\,.
$$
This implies that $A=B=0$, and so
$$
H^{2,0}_{\text{Dol}}=\left\lbrace 0\right\rbrace.
$$
 
 \medskip
\noindent
Therefore we proved the following
\begin{theorem}
Let $(X,J_{a,b,c},\omega_{a,b,c})$ be the family of almost-K\"ahler manifolds previously constructed. Then,
\begin{itemize}
\item $h^{1,0}_{\text{Dol}}=1$,
\item $h^{2,0}_{\text{Dol}}=0$,
\item $h^{3,0}_{\text{Dol}}=0$.
\end{itemize}
\end{theorem}

\section{An almost-complex structure with no compatible symplectic structures}\label{section:no-ak}

We will construct now an almost-complex structure $J$ on $X$ which does not admit any compatible symplectic structures. We set as a global co-frame of $(1,0)$-forms 
$$
\Phi^1:=e^1+ie^2\,,\qquad
\Phi^2:=e^3+ie^4\,,\qquad
\Phi^3:=e^5+ie^6\,,
$$
and the dual frame of $(1,0)$-vectors is given by
$$
W_1:=\frac{1}{2}\left(e_1-ie_2\right)\,,\qquad
W_2:=\frac{1}{2}\left(e_3-ie_4\right)\,,\qquad
W_3:=\frac{1}{2}\left(e_5-ie_6\right)\,.
$$
The complex structure equations become
$$
\left\lbrace
\begin{array}{lcl}
d\Phi^1 & =& 0\\
d\Phi^2 & = & \frac{i}{2}\Phi^{13}-\frac{1}{2}\Phi^{1\bar2}
+\frac{i}{2}\Phi^{3\bar1}-\frac{1}{2}\Phi^{\bar1\bar2}\\
d\Phi^3 & = & -\frac{1}{2}\Phi^{1\bar3}-\frac{1}{2}\Phi^{\bar1\bar3}\\
\end{array}
\right.\,.
$$

Notice that the almost-complex manifold just constructed does not admit any compatible symplectic structures. Indeed, by contradiction, if $(X,J)$ admits a compatible symplectic structure then, by a symmetrization process it also admits a compatible left-invariant symplectic structure. As noticed before, every left-invariant symplectic structure on $X$ is given by
$$
\omega_{a,b,c}=ae^{12}+be^{56}+c(e^{36}+e^{45})
$$
with $a,b,c\in\mathbb{R}$ and $a,c\neq 0$. Hence, by construction $J$ cannot be compatible with any of these symplectic structures.

We compute now the Hodge numbers $h^{p,0}_{\delbar}$, for $p=1,2,3$.

\subsection{Computations for $\mathcal{H}^{1,0}_{\delbar}$}

Let
$$
\psi=A\Phi^{1}+B\Phi^2+C\Phi^3
$$
with $A,B,C$ smooth functions on $X$, be an arbitrary $(1,0)$-form on $X$.
By degree reasons, $\psi$ is $\delbar$-harmonic if and only if $\delbar\psi=0$.
Using the structure equations we have that $\delbar\psi=0$ if and only if
$$
\left\lbrace
\begin{array}{lcl}
\bar W_1(A) & =& 0\\
\bar W_2(A)+\frac{1}{2}B & = & 0\\
\bar W_3(A)+\frac{1}{2}C & = & 0\\
\bar W_1(B) & = & 0\\
\bar W_2(B) & =& 0\\
\bar W_3(B) & =& 0\\
\bar W_1(C)-\frac{i}{2}B & = & 0\\
\bar W_2(C) & =& 0\\
\bar W_3(C) & =& 0\\
\end{array}
\right.\,.
$$
Then from $\bar W_1(B)=\bar W_2(B)=\bar W_3(B)=0$ we get with similar arguments used before that $B$ is constant. Hence
$$
(W_1\bar W_1+W_2\bar W_2+W_3\bar W_3)(C)=0
$$ 
and so $C$ is also constant. As a consequence, the same holds for $A$. Therefore, having $A$ constant, this implies that $B=C=0$. Therefore,
$$
B=0,\quad C=0, \quad A=\text{const}
$$
hence
$$
\mathcal{H}^{1,0}_{\delbar}=\left\langle \Phi^1\right\rangle
$$
and $h^{1,0}_{\delbar}=1$.\\

\subsection{Computations for $\mathcal{H}^{2,0}_{\delbar}$}

Let
$$
\psi=A\Phi^{12}+B\Phi^{13}+C\Phi^{23}
$$
with $A,B,C$ smooth functions on $X$, be an arbitrary $(2,0)$-form on $X$.
By degree reasons, $\psi$ is $\delbar$-harmonic if and only if $\delbar\psi=0$.
Using the structure equations we have that $\delbar\psi=0$ if and only if
$$
\left\lbrace
\begin{array}{lcl}
\bar W_1(A) & =& 0\\
\bar W_2(A) & = & 0\\
\bar W_3(A)-\frac{1}{2}C & = & 0\\
\bar W_1(B)-\frac{i}{2}A & = & 0\\
\bar W_2(B)+\frac{1}{2}C & =& 0\\
\bar W_3(B) & =& 0\\
\bar W_1(C) & = & 0\\
\bar W_2(C) & =& 0\\
\bar W_3(C) & =& 0\\
\end{array}
\right.\,.
$$
Then from $\bar W_1(C)=\bar W_2(C)=\bar W_3(C)=0$ we get with similar arguments used before that $C$ is constant. Hence
$(W_1\bar W_1+W_2\bar W_2+W_3\bar W_3)(A)=0$ and so $A$ is also constant. This implies that $C=0$ and therefore $B$ is constant leading to $A$ being zero. Namely
$$
A=0,\quad C=0, \quad B=\text{const}
$$
hence
$$
\mathcal{H}^{2,0}_{\delbar}=\left\langle \Phi^{13}\right\rangle
$$
and $h^{2,0}_{\delbar}=1$.\\

\subsection{Computations for $\mathcal{H}^{3,0}_{\delbar}$}

Let
$$
\psi=A\Phi^{123}
$$
with $A$ smooth function on $X$, be an arbitrary $(3,0)$-form on $X$.
By degree reasons, $\psi$ is $\delbar$-harmonic if and only if $\delbar\psi=0$.
Since $\Phi^{123}$ is $\delbar$-closed we have
that $\delbar\psi=0$ if and only if
$$
\bar W_1(A)=\bar W_2(A)=\bar W_3(A)=0
$$
hence $(W_1\bar W_1+W_2\bar W_2+W_3\bar W_3)(A)=0$ and so we have that $A$ is constant. Therefore,
$$
\mathcal{H}^{3,0}_{\delbar}(X)=\left\langle\Phi^{123}\right\rangle
$$
and $h^{3,0}_{\delbar}=1$.

\medskip
\noindent
Therefore, we just proved the following
\begin{theorem}
Let $(X,J,\omega)$ be the almost-Hermitian manifold previously constructed. Then,
\begin{itemize}
\item $h^{1,0}_{\delbar}=1$,
\item $h^{2,0}_{\delbar}=0$,
\item $h^{3,0}_{\delbar}=0$.
\end{itemize}
\end{theorem}

We compute now the numbers $h^{p,0}_{\bar\delta}$, for $p=1,2,3$.\\


\noindent
First of all, as noticed before, for bidegree reasons 
$$
\mathcal H^{1,0}_{\bar\delta}=\mathcal H^{1,0}_{\delbar}\,,
$$
hence we are left to compute $\mathcal H^{2,0}_{\bar\delta}$ and $\mathcal H^{3,0}_{\bar\delta}$.

\subsection{Computations for $\mathcal{H}^{2,0}_{\bar\delta}$}

It is immediate to see that
$$
\mathcal{H}^{2,0}_{\bar\delta}=\mathcal{H}^{2,0}_{\delbar}\cap
\text{Ker}\,(\mu^*)\,.
$$
Since $\mathcal{H}^{2,0}_{\delbar}=\langle\Phi^{13}\rangle$ we set
$\psi=A\Phi^{13}$ with $A\in\mathbb{C}$. Then, $\psi\in \text{Ker}\,(\mu^*)$ if and only if $\bar\mu*\psi=0$.
Since $*\psi=-A\frac{i}{2}\Phi^{123\bar2}$ and, by the structure equations
$$
\bar\mu\Phi^{23}=-\frac{1}{2}\Phi^{3\bar1\bar2}+\frac{1}{2}\Phi^{2\bar1\bar3}
$$
we have that
$$
\bar\mu*\psi=A\frac{i}{2}\Phi^1\wedge\bar\mu(\Phi^{23})\wedge
\Phi^{\bar2}=-A\frac{i}{4}\Phi^{12\bar1\bar2\bar3}.
$$
Then, $\bar\mu*\psi=0$ if and only if $A=0$. Therefore,
$$
\mathcal H^{2,0}_{\bar\delta}=\left\lbrace 0\right\rbrace
$$
and $h^{2,0}_{\bar\delta}=0$.

\subsection{Computations for $\mathcal{H}^{3,0}_{\bar\delta}$}

Clearly, as before
$$
\mathcal{H}^{3,0}_{\bar\delta}=\mathcal{H}^{3,0}_{\delbar}\cap
\text{Ker}\,(\mu^*)\,.
$$
Since $\mathcal{H}^{3,0}_{\delbar}=\langle\Phi^{123}\rangle$ we set
$\psi=A\Phi^{123}$ with $A\in\mathbb{C}$. Then, $\psi\in \text{Ker}\,(\mu^*)$ if and only if $\bar\mu*\psi=0$.
Since $*\psi=A\Phi^{123}$ and, by the structure equations
$$
\bar\mu*\psi=A\left(\frac{1}{2}\Phi^{13\bar1\bar2}-\frac{1}{2}\Phi^{12\bar1\bar3}\right).
$$
Then, $\bar\mu*\psi=0$ if and only if $A=0$. Therefore,
$$
\mathcal H^{3,0}_{\bar\delta}=\left\lbrace 0\right\rbrace
$$
and $h^{3,0}_{\bar\delta}=0$.\\

\medskip
\noindent
Therefore, we just proved the following
\begin{theorem}
Let $(X,J,\omega)$ be the almost-Hermitian manifold previously constructed. Then,
\begin{itemize}
\item $h^{1,0}_{\bar\delta}=1$,
\item $h^{2,0}_{\bar\delta}=0$,
\item $h^{3,0}_{\bar\delta}=0$.
\end{itemize}
\end{theorem}

\medskip
\noindent

We compute now the dimensions of the almost-complex Dolbeault cohomology groups $H^{p,0}_{\text{Dol}}$, for $p=1,2,3$.\\

As done above, notice that by \cite[Proposition 4.10]{cirici-wilson-1}, 
$$
H^{p,0}_{\text{Dol}}\simeq \mathcal{H}^{p,0}_{\delbar}\cap\text{Ker}\,\bar\mu.
$$

\subsection{Computations for $\mathcal{H}^{1,0}_{\text{Dol}}$, $\mathcal{H}^{2,0}_{\text{Dol}}$ and $\mathcal{H}^{3,0}_{\text{Dol}}$ }

Clearly, by the structure equations and by the previous computations

$$
H^{1,0}_{\text{Dol}}\simeq \mathcal{H}^{1,0}_{\delbar}\cap\text{Ker}\,\bar\mu=\left\langle\Phi^1\right\rangle\,.
$$
Now, since $\mathcal{H}^{2,0}_{\delbar}=\left\langle\Phi^{13}\right\rangle$
and by a direct computation $\bar\mu\Phi^{13}=\frac{1}{2}\Phi^{1\bar1\bar3}\neq 0$, one has that
$$
H^{2,0}_{\text{Dol}}=\left\lbrace 0\right\rbrace.
$$
Similarly, since $\mathcal{H}^{3,0}_{\delbar}=\left\langle\Phi^{123}\right\rangle$
and by a direct computation $\bar\mu\Phi^{123}\neq 0$, one has that
$$
H^{3,0}_{\text{Dol}}=\left\lbrace 0\right\rbrace.
$$

\medskip
\noindent
Therefore, we just proved the following
\begin{theorem}
Let $(X,J,\omega)$ be the almost-Hermitian manifold previously constructed.  Then,
\begin{itemize}
\item $h^{1,0}_{\text{Dol}}=1$,
\item $h^{2,0}_{\text{Dol}}=0$,
\item $h^{3,0}_{\text{Dol}}=0$.
\end{itemize}
\end{theorem}

\section{The Iwasawa manifold}\label{section:iwasawa}

We study now another $6$-dimensional example. Let $\mathbb{I}$ be the Iwasawa manifold defined as the quotient $\mathbb{I}:=\Gamma\backslash\mathbb{H}_3$ where
$$
\mathbb{H}_3:=\left\lbrace
\left[\begin{matrix}
1 & z_1 &   z_3 \\
0 & 1   & z_2\\
0 & 0 & 1\\
\end{matrix}\right]
\mid z_1,z_2,z_3\in\mathbb{C}
\right\rbrace
$$
and
$$
\Gamma:=\left\lbrace
\left[\begin{matrix}
1 & \gamma_1 &   \gamma_3 \\
0 & 1   & \gamma_2\\
0 & 0 & 1\\
\end{matrix}\right]
\mid \gamma_1,\gamma_2,\gamma_3\in\mathbb{Z}[\,i\,]
\right\rbrace\,.
$$
Then, setting $z_j=x_j+iy_j$, there exists a basis of
left-invariant $1$-forms $\left\lbrace e_i\right\rbrace$ on $\mathbb{I}$ given by
$$
\left\lbrace
\begin{array}{lcl}
e^1 & =& dx_1\\
e^2 & = & dy_1\\
e^3 & = & dx_2\\
e^4 & = & dy_2\\
e^5 & = & dx_3-x_1dx_2+y_1dy_2\\
e^6 & = & dy_3-x_1dy_2-y_1dx_2
\end{array}
\right.\,,
$$
and the dual basis is given by
$$
\left\lbrace
\begin{array}{lcl}
e_1 & =& \frac{\partial}{\partial x_1}\\
e_2 & = & \frac{\partial}{\partial y_1}\\
e_3 & = & \frac{\partial}{\partial x_2}+x_1\frac{\partial}{\partial x_3}+y_1\frac{\partial}{\partial y_3}\\
e_4 & = & \frac{\partial}{\partial y_2}-y_1\frac{\partial}{\partial x_3}+x_1\frac{\partial}{\partial y_3}\\
e_5 & = & \frac{\partial}{\partial x_3}\\
e_6 & = & \frac{\partial}{\partial y_3}
\end{array}
\right.\,.
$$
The following structure equations hold
$$
\left\lbrace
\begin{array}{lcl}
de^1 & =& 0\\
de^2 & = & 0\\
de^3 & = & 0\\
de^4 & = & 0\\
de^5 & = & -e^{13}+e^{24}\\
de^6 & = & -e^{14}-e^{23}
\end{array}
\right.\,.
$$
We define the almost-complex structure $J$ setting as global co-frame of $(1,0)$-forms 
$$
\varphi^1:=e^1+ie^6\,,\qquad
\varphi^2:=e^2+ie^5\,,\qquad
\varphi^3:=e^3+ie^4\,
$$
and let 
$$
V_1:=\frac{1}{2}\left(e_1-ie_6\right)\,,\qquad
V_2:=\frac{1}{2}\left(e_2-ie_5\right)\,,\qquad
V_3:=\frac{1}{2}\left(e_3-ie_4\right)\,
$$
 be the dual frame of vectors.
In particular, the complex structure equations become
$$
\left\lbrace
\begin{array}{lcl}
d\varphi^1 & =& -\frac{1}{4}\varphi^{13}-\frac{i}{4}\varphi^{23}+
\frac{1}{4}\varphi^{1\bar 3}-\frac{i}{4}\varphi^{2\bar 3}+
\frac{1}{4}\varphi^{3\bar 1}+\frac{i}{4}\varphi^{3\bar2}+
\frac{1}{4}\varphi^{\bar1\bar3}-\frac{i}{4}\varphi^{\bar2\bar3}\\[3pt]
d\varphi^2 & = & -\frac{i}{4}\varphi^{13}+\frac{1}{4}\varphi^{23}
-\frac{i}{4}\varphi^{1\bar 3}-\frac{1}{4}\varphi^{2\bar 3}+
\frac{i}{4}\varphi^{3\bar 1}-\frac{1}{4}\varphi^{3\bar2}-
\frac{i}{4}\varphi^{\bar1\bar3}-\frac{1}{4}\varphi^{\bar2\bar3}\\[3pt]
d\varphi^3 & = & 0\\
\end{array}
\right.\,.
$$
Notice that
$$
\omega:=\frac{i}{2}\sum_{j=1}^3\varphi^{j\bar j}
$$
is an almost-K\"ahler metric on $\mathbb{I}$, in particular $(J,\omega)$ is an almost-K\"ahler structure on $\mathbb{I}$.\\
\medskip

We compute now the Hodge numbers $h^{p,0}_{\delbar}$, for $p=1,2,3$.

\subsection{Computations for $\mathcal{H}^{1,0}_{\delbar}$}

Let
$$
\psi=A\varphi^{1}+B\varphi^2+C\varphi^3
$$
with $A,B,C$ smooth functions on $\mathbb{I}$, be an arbitrary $(1,0)$-form on $\mathbb{I}$.
By degree reasons, $\psi$ is $\delbar$-harmonic if and only if $\delbar\psi=0$.
Using the structure equations we have
that $\delbar\psi=0$ if and only if
$$
\left\lbrace
\begin{array}{lcl}
\bar V_1(A) & =& 0\\
\bar V_2(A) & = & 0\\
-\bar V_3(A)+\frac{1}{4}A-\frac{i}{4}B & = & 0\\
\bar V_1(B) & = & 0\\
\bar V_2(B) & =& 0\\
\bar V_3(B)+\frac{i}{4}A+\frac{1}{4}B & =& 0\\
-\bar V_1(C)+\frac{1}{4}A+\frac{i}{4}B & = & 0\\
-\bar V_2(C)+\frac{i}{4}A-\frac{1}{4}B & =& 0\\
\bar V_3(C) & =& 0\\
\end{array}
\right.\,.
$$
From $\bar V_1(A)=\bar V_2(A)=\bar V_1(B)=\bar V_2(B)=0$ we get that
$$
(V_1\bar V_1+V_2\bar V_2)(A)=0\qquad
\text{and}\qquad
(V_1\bar V_1+V_2\bar V_2)(B)=0
$$
and so $A=A(x_2,y_2)$ and $B=B(x_2,y_2)$ depend only on $x_2$ and $y_2$.\\
Hence, from the last three equations we obtain $(V_1\bar V_1+V_2\bar V_2+V_3\bar V_3)(C)=0$ implying that $C$ is constant.
Therefore, $A+iB=0$ giving
$$
-\bar V_3(A)+\frac{1}{2}A=0\qquad
\text{and}\qquad
-\bar V_3(B)-\frac{1}{2}B=0.
$$
We can expand in Fourier series and get
$$
A=\sum_{\lambda,\mu\in\mathbb{Z}}A_{\lambda\mu}e^{2\pi i(\lambda x_2+\mu y_2)},\quad
B=\sum_{\lambda,\mu\in\mathbb{Z}}B_{\lambda\mu}e^{2\pi i(\lambda x_2+\mu y_2)}
$$
with $A_{\lambda\mu},B_{\lambda\mu}$ constants for every $\lambda,\mu\in\mathbb{Z}$. Therefore, $\bar V_3(A)-\frac{1}{2}A=0$ gives
$$
\left(-\pi i\lambda+\pi\mu+\frac{1}{2} \right)A_{\lambda\mu}=0
$$
and since $\mu\in\mathbb{Z}$ we have that $A_{\lambda\mu}=0$ for every $\lambda,\mu\in\mathbb{Z}$. Hence,
$$
A=0\qquad
\text{and}\qquad
B=0.
$$
Therefore,
$$
A=0,\quad B=0, \quad C=\text{const}
$$
hence
$$
\mathcal{H}^{1,0}_{\delbar}=\left\langle \varphi^3\right\rangle
$$
and $h^{1,0}_{\delbar}=1$.\\

\subsection{Computations for $\mathcal{H}^{2,0}_{\delbar}$}

Let
$$
\psi=A\varphi^{12}+B\varphi^{13}+C\varphi^{23}
$$
with $A,B,C$ smooth functions on $\mathbb{I}$, be an arbitrary $(2,0)$-form on $\mathbb{I}$.
By degree reasons, $\psi$ is $\delbar$-harmonic if and only if $\delbar\psi=0$.
Using the structure equations we have
that $\delbar\psi=0$ if and only if
$$
\left\lbrace
\begin{array}{lcl}
\bar V_1(A) & =& 0\\
\bar V_2(A) & = & 0\\
\bar V_3(A) & = & 0\\
\bar V_1(B)-\frac{i}{4}A & = & 0\\
\bar V_2(B)+\frac{1}{4}A & =& 0\\
\bar V_3(B) -\frac{1}{4}B+\frac{i}{4}C& =& 0\\
\bar V_1(C)+\frac{1}{4}A & = & 0\\
\bar V_2(C) +\frac{i}{4}A& =& 0\\
\bar V_3(C)+\frac{i}{4}B+\frac{1}{4}C & =& 0\\
\end{array}
\right.\,.
$$
With similar arguments used above we have that $A=\text{const}$, $B=B(x_2,y_2)$ and $C=C(x_2,y_2)$. In particular, since $\bar V_1(B)=0$ we get that $A=0$.
Therefore, from
$$
\bar V_3(B) -\frac{1}{4}B+\frac{i}{4}C = 0\qquad\text{and}\qquad
\bar V_3(C)+\frac{i}{4}B+\frac{1}{4}C  = 0
$$
we obtain $\bar V_3(B-iC)=0$ hence, $B-iC=\text{const}=:k$.
In particular,
$$
\bar V_3(B) -\frac{1}{4}k=0
$$
and so $B$ is constant implying that also $C$ is constant.
Therefore, $k=0$ giving $B=iC$.\\
Therefore,
$$
A=0,\quad B=iC=\text{const}, 
$$
hence
$$
\mathcal{H}^{2,0}_{\delbar}=\left\langle i\varphi^{13}+\varphi^{23}\right\rangle
$$
and $h^{2,0}_{\delbar}=1$.\\

\subsection{Computations for $\mathcal{H}^{3,0}_{\delbar}$}

Let
$$
\psi=A\varphi^{123}
$$
with $A$ smooth function on $\mathbb{I}$, be an arbitrary $(3,0)$-form on $\mathbb{I}$.
By degree reasons, $\psi$ is $\delbar$-harmonic if and only if $\delbar\psi=0$.
Hence $\delbar\psi=0$ if and only if
$$
\bar V_1(A)=\bar V_2(A)=\bar V_3(A)=0
$$
hence $(V_1\bar V_1+V_2\bar V_2+V_3\bar V_3)(A)=0$ and, since $V_1\bar V_1+V_2\bar V_2+V_3\bar V_3$ is an elliptic operator we have that $A$ is constant. Therefore,
$$
\mathcal{H}^{3,0}_{\delbar}(X)=\left\langle\varphi^{123}\right\rangle
$$
and $h^{3,0}_{\delbar}=1$.

Therefore, we just proved the following
\begin{theorem}
Let $(\mathbb{I},J,\omega)$ be the almost-K\"ahler Iwasawa manifold constructed above. Then,
\begin{itemize}
\item $h^{1,0}_{\delbar}=1$,
\item $h^{2,0}_{\delbar}=1$
\item $h^{3,0}_{\delbar}=1$.
\end{itemize}
\end{theorem}

We compute now the numbers $h^{p,0}_{\bar\delta}$, for $p=1,2,3$.


\noindent
\medskip

First of all, as noticed before, for bidegree reasons 
$$
\mathcal H^{1,0}_{\bar\delta}=\mathcal H^{1,0}_{\delbar}\,,
$$
hence we are left to compute $\mathcal H^{2,0}_{\bar\delta}$ and $\mathcal H^{3,0}_{\bar\delta}$.

\subsection{Computations for $\mathcal{H}^{2,0}_{\bar\delta}$}

It is immediate to see that
$$
\mathcal{H}^{2,0}_{\bar\delta}=\mathcal{H}^{2,0}_{\delbar}\cap
\text{Ker}\,(\mu^*)\,.
$$
Since 
$$
\mathcal{H}^{2,0}_{\delbar}=\langle i\varphi^{13}+\varphi^{23}\rangle
$$
 we set
$$
\psi=A(i\varphi^{13}+\varphi^{23})
$$
 with $A\in\mathbb{C}$. 
Then, $\psi\in \text{Ker}\,(\mu^*)$ if and only if $\bar\mu*\psi=0$.
Since $*\psi=A\cdot\text{const}\cdot(-i\varphi^{123\bar2}+\varphi^{123\bar1})$ and by the structure equations we have that
$$
\bar\mu\varphi^{123\bar 2}=-\frac{1}{4}\varphi^{23\bar1\bar2\bar3}-
\frac{i}{4}\varphi^{13\bar1\bar2\bar3}
$$
and
$$
\bar\mu\varphi^{123\bar 1}=-\frac{i}{4}\varphi^{23\bar1\bar2\bar3}+
\frac{1}{4}\varphi^{13\bar1\bar2\bar3}
$$
we get that
$$
\bar\mu*\psi=0
$$
Therefore,
$$
\mathcal H^{2,0}_{\bar\delta}=\mathcal{H}^{2,0}_{\delbar}=\langle i\varphi^{13}+\varphi^{23}\rangle
$$
and $h^{2,0}_{\bar\delta}=1$.

\subsection{Computations for $\mathcal{H}^{3,0}_{\bar\delta}$}

Clearly, as before
$$
\mathcal{H}^{3,0}_{\bar\delta}=\mathcal{H}^{3,0}_{\delbar}\cap
\text{Ker}\,(\mu^*)\,.
$$
Since $\mathcal{H}^{3,0}_{\delbar}=\langle\varphi^{123}\rangle$ we set
$\psi=A\varphi^{123}$ with $A\in\mathbb{C}$. Then, $\psi\in \text{Ker}\,(\mu^*)$ if and only if $\bar\mu*\psi=0$.
Since $*\psi=A\cdot\text{const}\cdot\varphi^{123}$ and, by the structure equations
$$
\bar\mu*\psi=A\cdot\text{const}\cdot\left(\frac{1}{4}\varphi^{23\bar1\bar3}-
\frac{i}{4}\varphi^{23\bar2\bar3}+
\frac{i}{4}\varphi^{13\bar1\bar3}+
\frac{1}{4}\varphi^{13\bar2\bar3}\right).
$$
Then, $\bar\mu*\psi=0$ if and only if $A=0$. Therefore,
$$
\mathcal H^{3,0}_{\bar\delta}=\left\lbrace 0\right\rbrace
$$
and $h^{3,0}_{\bar\delta}=0$.

\medskip
\noindent
Therefore, we just proved the following
\begin{theorem}
Let $(\mathbb I,J,\omega)$ be the almost-K\"ahler Iwasawa manifold previously constructed. Then,
\begin{itemize}
\item $h^{1,0}_{\bar\delta}=1$,
\item $h^{2,0}_{\bar\delta}=1$,
\item $h^{3,0}_{\bar\delta}=0$.
\end{itemize}
\end{theorem}

We compute now the dimensions of the almost-complex Dolbeault cohomology groups $H^{p,0}_{\text{Dol}}$, for $p=1,2,3$.\\

\medskip


As done above, notice that by \cite[Proposition 4.10]{cirici-wilson-1}, 
$$
H^{p,0}_{\text{Dol}}\simeq \mathcal{H}^{p,0}_{\delbar}\cap\text{Ker}\,\bar\mu.
$$

\subsection{Computations for $\mathcal{H}^{1,0}_{\text{Dol}}$, $\mathcal{H}^{2,0}_{\text{Dol}}$ and $\mathcal{H}^{3,0}_{\text{Dol}}$ }

Clearly, by the structure equations and by the previous computations

$$
H^{1,0}_{\text{Dol}}\simeq \mathcal{H}^{1,0}_{\delbar}\cap\text{Ker}\,\bar\mu=\left\langle\varphi^3\right\rangle\,.
$$
Now, since $\mathcal{H}^{2,0}_{\delbar}=\left\langle i\varphi^{13}+\varphi^{23}\right\rangle$
and by a direct computation $\bar\mu(i\varphi^{13}+\varphi^{23})= 0$, one has that
$$
H^{2,0}_{\text{Dol}}=\left\langle i\varphi^{13}+\varphi^{23}\right\rangle.
$$
Since $\mathcal{H}^{3,0}_{\delbar}=\left\langle\varphi^{123}\right\rangle$
and by a direct computation $\bar\mu\varphi^{123}\neq 0$, one has that
$$
H^{3,0}_{\text{Dol}}=\left\lbrace 0\right\rbrace.
$$

\medskip
\noindent
In particular, we have the following
\begin{theorem}
Let $(\mathbb I,J,\omega)$ be the almost-K\"ahler Iwasawa manifold previously constructed. Then,
\begin{itemize}
\item $h^{1,0}_{\text{Dol}}=1$,
\item $h^{2,0}_{\text{Dol}}=1$,
\item $h^{3,0}_{\text{Dol}}=0$.
\end{itemize}
\end{theorem}

\section{Obstructions to the existence of a compatible symplectic structure on an almost-complex manifold}

Let $(X,J)$ be an almost-complex manifold and fix a Hermitian metric $g$ with fundamental form $\omega$. Then, setting $\bar\delta:=\delbar+\mu$ and $\delta:=\del+\bar\mu$ one can consider
the following differential operators 
$$
\Delta_{\bar\delta}:=\bar\delta\bar\delta^*+\bar\delta^*\bar\delta\,,
$$
$$
\Delta_{\delta}:=\delta\delta^*+\delta^*\delta\,.
$$
In \cite{tardini-tomassini} we studied Hodge theory for such operators, and even though they do not coincide in general, as a consequence of the almost-K\"ahler identities, if $(X,J,g,\omega)$ is an almost-K\"ahler manifold, then
$\Delta_{\bar\delta}$ and $\Delta_{\delta}$ are related by 
$$
\Delta_{\bar\delta}=\Delta_{\delta}\,.
$$
In particular, their spaces of harmonic forms coincide, i.e. $\mathcal{H}^{\bullet}_{\delta}(X)=\mathcal{H}^{\bullet}_{\bar\delta}(X)$\,.\\
We can use now this result to prove an obstruction to the existence of a compatible symplectic structure on an almost-complex manifold.

\begin{theorem}
Let $(X,J)$ be a compact almost-complex manifold. Suppose that there exists $\varphi\in A^{1,0}(X)$ such that $\delbar\varphi=0$ and $d\varphi\neq 0$. Then, there exists no compatible symplectic structure on $(X,J)$.
\end{theorem}
\begin{proof}
Since, $\delbar\varphi=0$ then, for degree reasons $\varphi\in\text{Ker}\,\Delta_{\bar\delta}$ for any arbitrary Hermitian metric. However, since $d\varphi\neq 0$ then, for any fixed Hermitian metric, $\varphi\notin\text{Ker}\,\Delta_{\delta}$. Namely, $\Delta_{\bar\delta}\neq \Delta_{\delta}$ and the thesis follows, since, by \cite{tardini-tomassini} on almost-K\"ahler manifolds
$\Delta_{\bar\delta}= \Delta_{\delta}$.
\end{proof}

An immediate corollary is the following

\begin{cor}\label{cor:no-ak}
Let $(X,J)$ be a compact almost-complex manifold such that there exists a global co-frame of $(1,0)$-forms $\left\lbrace\varphi^i\right\rbrace$ such that, there exists an index $j$ with
$$
d\varphi^j\in A^{2,0}(X)\oplus A^{0,2}(X)
$$
and $d\varphi^j\neq 0$.
Then, there exists no compatible symplectic structure on $(X,J)$.
\end{cor}

We apply this result to the following example.
\begin{ex}
Let $\mathbb{I}$ be the Iwasawa manifold defined as the quotient $\mathbb{I}:=\Gamma\backslash\mathbb{H}_3$ where
$$
\mathbb{H}_3:=\left\lbrace
\left[\begin{matrix}
1 & z_1 &   z_3 \\
0 & 1   & z_2\\
0 & 0 & 1\\
\end{matrix}\right]
\mid z_1,z_2,z_3\in\mathbb{C}
\right\rbrace
$$
and
$$
\Gamma:=\left\lbrace
\left[\begin{matrix}
1 & \gamma_1 &   \gamma_3 \\
0 & 1   & \gamma_2\\
0 & 0 & 1\\
\end{matrix}\right]
\mid \gamma_1,\gamma_2,\gamma_3\in\mathbb{Z}[\,i\,]
\right\rbrace\,.
$$
Set $\psi^1:=d\bar z_1$, $\psi^2:=d\bar z_2$ $\psi^3:=d\bar z_3-z_1dz_2$.
Hence, the structure equations are
$$
d\psi^1=0,\quad d\psi^2=0,\quad d\psi^3=-\psi^{\bar 1\bar 2},
$$
therefore, by Corollary \ref{cor:no-ak} the Iwasawa manifold with this almost-complex structure does not admit any compatible symplectic structure.
\end{ex}

Clearly, the converse implication does not hold as we have seen in Section \ref{section:no-ak}.


\end{document}